\documentclass[preprint]{elsarticle}

\usepackage[usenames,dvipsnames,svgnames,table]{xcolor}
\usepackage[colorlinks,citecolor=blue,urlcolor=blue]{hyperref}

\usepackage{hyperref}

\usepackage{amsmath}
\usepackage{amsthm}
\usepackage{mathrsfs}
\usepackage{amssymb}
\usepackage{amsfonts}
\usepackage{graphicx}
\usepackage{stmaryrd}
\usepackage{enumitem}
\usepackage{tikz}
\usepackage{subcaption}
\usepackage{cleveref}

\journal{Discrete Mathematics}
\newtheorem{thm}{Theorem}[section]

\newtheorem{lem}[thm]{Lemma}
\newtheorem{cor}[thm]{Corollary}

\theoremstyle{definition}
\newtheorem{defn}[thm]{Definition}

\newtheorem{case}{Case}
\newtheorem{subcase}{Subcase}
\newtheorem{subsubcase}{Subcase}
\newtheorem{subsubsubcase}{Subcase}

\newtheorem*{case*}{Case}

\theoremstyle{remark}

\newtheorem{rmk}[thm]{Remark}

\newtheorem*{rmk*}{Remark}

\numberwithin{equation}{section}
\numberwithin{subcase}{case}
\numberwithin{subsubcase}{subcase}
\numberwithin{subsubsubcase}{subsubcase}

\newcommand{\bm}[1]{{\mbox{\boldmath$#1$}}}

\DeclareMathOperator{\height}{height}
\DeclareMathOperator{\sdepth}{sdepth}
\DeclareMathOperator{\tdepth}{tdepth}

\begin{document}

\begin{frontmatter}
\title{A Variant of the Stanley Depth for Multisets}

\author[mymainaddress]{Yinghui Wang}
\ead{yinghui@alum.mit.edu}

\address[mymainaddress]{Department of Mathematics, Massachusetts Institute of Technology, \\Cambridge, MA 02139, United States}


\begin{abstract}
We define and study a variant 
of the \emph{Stanley depth} which we call \emph{total depth} for partially ordered sets (posets). 
This total depth is the most natural variant 
of Stanley depth from $\llbracket S_k\rrbracket$ -- the poset of nonempty subsets of $\{1,2,\dots,k\}$  ordered by inclusion -- to any finite poset. 
In particular, the total depth can be defined for the poset of nonempty submultisets of a multiset ordered by inclusion, which corresponds to a product of chains with the bottom element deleted. 
We show that the total depth agrees with Stanley depth for $\llbracket S_k\rrbracket$  but not for such posets in general. 
We also prove that the total depth of the product of chains $\bm{n}^k$ with the bottom element deleted is $(n-1)\lceil{k/2}\rceil$, which generalizes a result of Bir{\'{o}}, Howard, Keller, Trotter, and Young (2010). 
Further, we provide upper and lower bounds for a general multiset and find the total depth for any multiset with at most five distinct elements. 
In addition, we can determine the total depth for any multiset with $k$ distinct elements if we know all the interval partitions of $\llbracket S_k\rrbracket$.
\end{abstract}


\end{frontmatter}


\section{Introduction}
In \cite{S82} Stanley defined what is now called the \emph{Stanley depth} of a finitely-generated $\mathbb{Z}^n$-graded module over a commutative ring (see \cite{H} for a survey). 
Herzog et al. \cite{HVZ} established a combinatorial definition for the Stanley depth of certain modules in terms of \emph{partially ordered sets (posets)} that are the direct products of chains with the bottom element deleted, or equivalently, the sets of nonempty submultisets of multisets ordered by inclusion. 
In particular, Bir{\'{o}} et al. \cite{B} showed that the Stanley depth of $\llbracket S_k\rrbracket$ -- the poset of nonempty subsets of $S_k:=\{1,2,\dots,k\}$, ordered by inclusion -- is $\lceil k/2\rceil$.
Shen \cite{S} extended this result by finding that
Stanley depth of a complete intersection monomial ideal minimally-generated by $m$ monomials is
$n - \lfloor m/2 \rfloor$, and this was later proved to be a lower bound for the Stanley depth of an $m$-generated squarefree monomial ideal by Keller and Young \cite{KY}. 

Although the Stanley depth is an interesting invariant of the posets just mentioned, there is no natural way to extend the definition to any (finite) poset. 
In this paper we will define and study a variant 
of Stanley depth which we call \emph{total depth}.
The total depth of a finite poset $P$ is the largest integer $d$ such that there is a partition of $P$ 
into intervals whose top elements all have height at least $d$, where the {\it height} of an element $Z\in P$ has a standard definition -- the maximal size of a chain in $P$ starting with $Z$ (Definition \ref{def:height}). 
For instance, if $P$ has a unique maximal element $\hat{1}$ and the height of $\hat{1}$ is $h$, then the total depth of $P$ is $h$ if and only if $P$ has a unique minimal element.

This total depth agrees with Stanley depth for $\llbracket S_k\rrbracket$ but not for the more general product of chains with the bottom element deleted (Section \ref{sec:comparison}). 
The total depth is the most natural variant 
of Stanley depth from $\llbracket S_k\rrbracket$ to any finite poset. 
It appears to be an interesting combinatorial invariant of a poset for its own sake, and it seems to be quite challenging to compute in general.

This paper begins with the definitions and comparisons of the Stanley depth and the total depth in Section \ref{sec:def}. 
Section \ref{sec:sdepthS} presents our main approach and shows that the Stanley depths of a multiset and its base set coincide.
Section \ref{sec:optimal} proves that one of the interval partitions with the maximal total depth has a special form \eqref{eq:f}.
Sections \ref{sec:n^k} and \ref{sec:bounds} derive upper and lower bounds for the total depth, and thus show that the total depth of the product of chains $\bm{n}^k$ with the bottom element deleted is $(n-1)\lceil{k/ 2}\rceil$ (Corollary \ref{coro:n^k}). 
Finally, Section \ref{sec:k=5} determines the total depths for multisets with at most five distinct elements.

\section{Definitions and Comparisons}
\label{sec:def}

\subsection{Stanley Depth}
We first recall the definition of Stanley depth from \cite{P}.

\begin{defn}
Let $\mathbb{K}$ be a field and $\mathcal{A}:=\mathbb{K}[x_1,x_2,\dots, x_k]$ the $\mathbb{K}$-algebra of polynomials over $\mathbb{K}$ in $x_1,x_2,\dots,x_k$. 

\begin{enumerate}[leftmargin=*,topsep=-0.5ex,itemsep=-0.5ex,partopsep=1ex,parsep=1ex]
\item 
A \emph{monomial} in  $\mathcal{A}$ is a product $x_1^{a_1} x_2^{a_2} \cdots x_k^{a_k}$ with $a_j$'s nonnegative integers; 

\item 
A \emph{monomial ideal} $I$ of $\mathcal{A}$ is an ideal generated by monomials in $\mathcal{A}$\,;

\item 
A \emph{Stanley space} of dimension $d$ is a $\mathbb{K}$-subspace of $\mathcal{A}$ of form $x_1^{a_1}\! x_2^{a_2}\! \cdots x_k^{a_k}\mathbb{K}[\mathbf{Z}]$, where $\mathbf{Z}$ is a $d$-element subset of $\{x_1,x_2,\dots, x_k\}$; 

\item 
A \emph{Stanley decomposition} $\mathcal{D}$ of $I$ is a decomposition of $I$ as a finite direct sum of Stanley spaces; 

\item 
The \emph{Stanley depth} of $\mathcal{D}$, denoted by ${\sdepth}\,\mathcal{D}$, is the {\it minimal} dimension of these Stanley spaces, and the \emph{Stanley depth} of $I$ is
\begin{equation*}
{\sdepth}\,I\ :=\ \max_{\mathcal{D}} {\sdepth}\,\mathcal{D},
\end{equation*}
where the maximum is taken over all the Stanley decompositions $\mathcal{D}$ of $I$.
\end{enumerate}
\end{defn}

In a combinatorial view, each monomial $x_1^{a_1} x_2^{a_2} \cdots x_k^{a_k}$ corresponds to a multiset $\{1^{a_1}, 2^{a_2}, \dots, k^{a_k}\}$ consisting of $a_j$ $j$'s, $j=1,2,\dots,k$, and each monomial ideal induces a set of such multisets. 
In the case that a monomial ideal $I^{\star}$ induces the set of \emph{all} nonempty submultisets of the  multiset
\begin{equation}\label{eq:S}
\{1^{n_1}, 2^{n_2}, \dots, k^{n_k}\}, 
\end{equation}
denoted by $S$ throughout this paper, the Stanley depth of $I^{\star}\!$ has a combinatorial interpretation in terms of the poset $\llbracket S\rrbracket$ of nonempty submultisets of $S$ ordered by inclusion (Definition \ref{def:I} below). 
Here, a \emph{submultiset} of $S$ is a multiset $\{1^{a_1}, 2^{a_2}, \dots, k^{a_k}\}$ with $a_j\le n_j,\ 1\le j\le k$. 

On the other hand, the submultisets of this $S$ correspond bijectively to the elements of a {\it product of chains} 
\begin{equation}\label{eq:U}
U=(\bm{n_1+1})\times (\bm{n_2+1})\times \cdots \times (\bm{n_k+1}),
\end{equation}
where $\bm{n+1}$ denotes the $(n+1)$-element chain $0<1<\cdots <n$. 
In particular, the element $(a_1,a_2,\dots,a_k)\in U$ corresponds to the submultiset $a=\{1^{a_1},2^{a_2},$ $\dots, k^{a_k}\}$ of $S$. 
For convenience, we denote this  $a$  by $(a_1,a_2,\dots,a_k)$, or abbreviated to $a_1a_2\dots a_k$, 
and  the {\it bottom element} $(0,0,\dots,0)$ of $U$ by $\bm{0}$.
Then $I^{\star}\!$ corresponds to $U\,\backslash \{\bm{0}\}$, the product of chains with the bottom element deleted.

\smallskip
The authors of \cite{HVZ} found a correspondence between 
the Stanley decompositions of $I^{\star}\!$ and the interval partitions (defined in Definition \ref{def:pi} below) of $\llbracket S\rrbracket$, which allowed them to compute the Stanley depths of these Stanley decompositions and of $I^{\star}\!$ in terms of these interval partitions, and thus to define the Stanley depths for interval partitions of $\llbracket S\rrbracket$ and for $\llbracket S\rrbracket$ (Definition \ref{def:I}). 
In fact, they \cite{HVZ} showed these results not only for $I^{\star}\!$ but also for the more general $\mathbb{Z}^n$-graded modules $I/I'$,  where $I' \subset I$ are monomial ideals in $S$.


\begin{defn}\label{def:pi}
For two sets $X\le Y$ in a poset $P$, an \emph{interval} $[X,Y]$ is defined as 
$\{Z\in P: X\le Z\le Y\}$. 
An \emph{interval partition} $\pi$ of $P$ is a partition of $P$ into nonempty pairwise disjoint intervals.
\end{defn}

In our case, the poset $P=\llbracket S\rrbracket$. 
Here and throughout this paper, 
we denote by $S$ the multiset 
of \eqref{eq:S}, 
by $S_k$ the base set $\{1,2,\dots,k\}$ of $S$,
by $\llbracket n_1,n_2,\dots,n_k \rrbracket$ the poset $\llbracket S\rrbracket$, 
and by $x_j$ (resp., $y_j$) the $j$-th components of $X$ (resp., $Y$) in $\llbracket S\rrbracket$, $j=1,2,\dots,k$.
Then $X\le Y$ is equivalent to $x_j\le y_j$ for all $j$, and $X< Y$ if further $x_{j'}< y_{j'}$ for some $j'$.

Fig.\,\ref{fig:1} illustrates the {\it Hasse diagrams} of two posets  $\llbracket S_3 \rrbracket=\llbracket\{1,2,3\}\rrbracket=\llbracket1,1,$ $1\rrbracket$ and $\llbracket\{1^3,2^2\}\rrbracket=\llbracket\{1,1,1,2,2\}\rrbracket=\llbracket2,3\rrbracket$ and one interval partition of each poset.
Note that according to the notation below \eqref{eq:U}, the labels in Fig.\,\ref{fig:1}(a) are $(1,1,1), (1,1,0), \dots, (0,0,1)$, or abbreviated to $111,110,\dots,001$. 
Similarly, the labels in Fig.\,\ref{fig:1}(b) are 
abbreviated to $23,22,\dots,01$. 
  
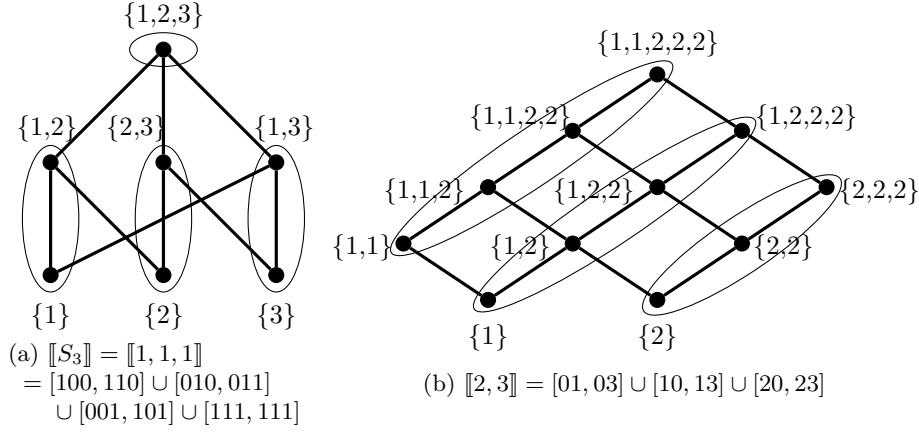
\begin{figure}[!h]
\caption{Posets and Interval Partitions}\label{fig:1}
\begin{subfigure}{.339\textwidth}
\centering
\begin{tikzpicture}[scale=0.75,roundnode/.style={circle, draw, fill, inner sep=2pt}]
\node [roundnode][label={[shift={(0,-0.95)}]\{1\}}] (v100) at (-2,-1) {};
\node [roundnode][label={[shift={(0,-0.95)}]\{2\}}] (v010) at (0,-1) {};
\node [roundnode][label={[shift={(0,-0.95)}]\{3\}}] (v001) at (2,-1) {};
\node [roundnode][label={[shift={(-0.05,0.05)}]\{1,2\}}] (v110) at (-2,1) {};
\node [roundnode][label={[shift={(0.1,0.05)}]\{1,3\}}] (v101) at (2,1) {};
\node [roundnode][label={[shift={(-0.36,0.05)}]\{2,3\}}] (v011) at (0,1) {};
\node [roundnode][label={[shift={(0,0.05)}]\{1,2,3\}}] (v111) at (0,3) {};

\draw [very thick] (v100) edge (v110);
\draw [very thick] (v100) edge (v101);
\draw [very thick] (v010) edge (v110);
\draw [very thick] (v010) edge (v011);
\draw [very thick] (v001) edge (v101);
\draw [very thick] (v001) edge (v011);
\draw [very thick] (v110) edge (v111);
\draw [very thick] (v011) edge (v111);
\draw [very thick] (v101) edge (v111);

\draw  (-2,0) ellipse (0.5 and 1.3);
\draw  (0,0) ellipse (0.5 and 1.3);
\draw (2,0) ellipse (0.5 and 1.3);
\draw  (0,3) ellipse (0.6 and 0.3);
\end{tikzpicture}
\caption{
$\llbracket S_3 \rrbracket=\llbracket 1,1,1\rrbracket  \\ \hphantom{\,\,}  =  [100,110]\cup[010, 011]
\\ \hphantom{a\,=} 
 \cup[001,101]\cup[111,111] $}
\label{fig:sfig1}  
\end{subfigure}
\begin{subfigure}{.661\textwidth}
\centering
\begin{tikzpicture}[scale=0.75,roundnode/.style={circle, draw, fill, inner sep=2pt}]
\node [roundnode][label={[shift={(0,0)}]\{1,1,2,2,2\}}] (v23) at (0,3) {};
\node [roundnode][label={[shift={(0.85,-0.23)}]\{1,2,2,2\}}] (v13) at (1.5,2) {};
\node [roundnode][label={[shift={(-0.7,-0.23)}]\{1,1,2,2\}}] (v22) at (-1.5,2) {};
\node [roundnode][label={[shift={(0.7,-0.46)}]\{2,2,2\}}] (v03) at (3,1) {};
\node [roundnode][label={[shift={(-0.85,-0.46)}]\{1,2,2\}}] (v12) at (0,1) {};
\node [roundnode][label={[shift={(-0.85,-0.46)}]\{1,1,2\}}] (v21) at (-3,1) {};
\node [roundnode][label={[shift={(0.55,-0.46)}]\{2,2\}}] (v02) at (1.5,0) {};
\node [roundnode][label={[shift={(-0.7,-0.46)}]\{1,2\}}] (v11) at (-1.5,0) {};
\node [roundnode][label={[shift={(-0.55,-0.46)}]\{1,1\}}] (v20) at (-4.5,0) {};
\node [roundnode][label={[shift={(0,-0.9)}]\{1\}}] (v10) at (-3,-1) {};
\node [roundnode][label={[shift={(0,-0.9)}]\{2\}}] (v01) at (0,-1) {};

\draw [very thick] (v01) edge (v02);
\draw [very thick] (v01) edge (v11);
\draw [very thick] (v02) edge (v03);
\draw [very thick] (v02) edge (v12);
\draw [very thick] (v11) edge (v12);
\draw [very thick] (v11) edge (v21);
\draw [very thick] (v20) edge (v21);
\draw [very thick] (v03) edge (v13);
\draw [very thick] (v10) edge (v11);
\draw [very thick] (v10) edge (v20);
\draw [very thick] (v12) edge (v13);
\draw [very thick] (v12) edge (v22);
\draw [very thick] (v21) edge (v22);
\draw [very thick] (v13) edge (v23);
\draw [very thick] (v22) edge (v23);

\draw [rotate=34] (-1,2.5) ellipse (3 and 0.5);
\draw [rotate=34] (-0.34,0.85) ellipse (3 and 0.5);
\draw [rotate=34] (1.23,-0.86) ellipse (2.1 and 0.5);
\end{tikzpicture}
\caption{$\llbracket2,3\rrbracket =  [01,03]\cup[10,13]\cup[20,23] $  }
\label{fig:sfig2}
\end{subfigure}
\end{figure}


\medskip
Then the Stanley depth of $\llbracket S\rrbracket$  can be interpreted 
as follows \cite{HVZ}.

\begin{defn}\label{def:I}
Let $I^{\star}\!$ be the monomial ideal corresponding to $\llbracket S\rrbracket$ 
and $\pi=\bigcup_i\, [X_i,Y_i]$ be the interval partition of $\llbracket S\rrbracket$ corresponding to a Stanley decomposition $\mathcal{D}$ of $I^{\star}$, where $U$ and $S$ are defined by \eqref{eq:U} and \eqref{eq:S}, respectively, and  $X_i=(x_{i1},x_{i2},\dots, x_{ik})$ and $Y_i=(y_{i1},y_{i2},\dots, y_{ik})$ with $x_{ij},y_{ij} \in \{0,1,\dots, n_j\}$, $j=1,2,\dots,k$.
We define the {\it Stanley depth} of $\pi$ as
\begin{equation}\label{eq:sdepth pi}
{\sdepth}\,\pi := {\sdepth}\,\mathcal{D} = \min_{i} \#\{j: y_{ij}=n_j\}=\min_{[X,Y]\in \pi} \#\{j: y_{j}=n_j\},
\end{equation}
and the {\it Stanley depth} of $\llbracket S\rrbracket$ as (see \cite[pp.\,476--477]{B})
\begin{equation}\label{eq:sdepthS}
{\sdepth}\,\llbracket S\rrbracket := {\sdepth}\,I^{\star}\! =\! \max_{\pi\in \mathfrak{P}(S)}\! {\sdepth}\,\pi= \! \max_{\pi\in \mathfrak{P}(S)} \min_{[X,Y]\in \pi}\! \#\{j\!: y_{j}=n_j\},
\end{equation}
where $\mathfrak{P}(S)$ is the set of interval partitions of $\llbracket S\rrbracket$.
\end{defn}

For example, the Stanley depths of the partitions in Fig.\,\ref{fig:1} are 
(a) $\min\{2,2,2,$ $3\} =2$; 
(b) $\min\{1,1,2\} =1$.

\subsection{Total Depth}
The Stanley depth of $\llbracket S_k\rrbracket$ can be recast in terms of the heights of its elements (defined below). This inspires us to extend Stanley depth to a general finite poset and thus define the total depth.

\begin{defn}\label{def:height}
Let $P$ be any finite poset. We define the {\it height} of an element $Z\in P$ as the maximal size of a chain $Z=Z_1>Z_2>\cdots >Z_n$ in $P$, denoted by ${\height}\,Z$ = $n$. 
\end{defn}

When $Z=(z_1,z_2,\dots,z_k)\in \llbracket S\rrbracket$ with $S$  the multiset of \eqref{eq:S}, by definition
\begin{equation}\label{eq:height}
{\height}\,Z = z_1+z_2+\cdots+z_k\,.
\end{equation}
In particular, when $S$ is the {\it normal} set $S_k=\{1,2,\dots,k\}$, i.e., when all $n_j$'s equal $1$, we have $z_j\in\{0,1\}$ and thus
\begin{equation}\label{eq:depthS_k}
{\height}\,Z = z_1+z_2+\cdots+z_k=\#\{j: z_{j}=1\}=\#\{j: z_{j}=n_j\}.
\end{equation}
Applying to $Y_i=
(y_{i1},y_{i2},\dots,y_{ik})$ and combining with \eqref{eq:sdepthS} leads to
\begin{equation}\label{eq:sdepth S_k}
{\sdepth}\,\llbracket S_k\rrbracket  =\max_{\pi\in \mathfrak{P}(S_k)} \min_{[X,Y]\in \pi} {\height}\,Y.
\end{equation}
Thus a natural variant 
of Stanley depth to any finite poset arises. 

\begin{defn}\label{def:tdepth}
Given a finite poset $P$, we define the \emph{total depth} of an interval partition $\pi$ of $P$ as
\begin{equation}\label{eq:tdepth pi}
{\tdepth}\,\pi:= \min_{[X,Y]\in \pi} {\height}\,Y,
\end{equation}
and the \emph{total depth} of $P$ as
\begin{equation}\label{eq:tdepth P}
{\tdepth}\,P:= \max_\pi {\tdepth}\,\pi 
=\max_\pi \min_{[X,Y]\in \pi} {\height}\,Y,
\end{equation}
where the maximum is taken over all the interval partitions $\pi$ of $P$. 
\end{defn}

In this paper we are interested in the case $P=\llbracket S\rrbracket$.
For example, the total depths of the partitions in Fig.\,\ref{fig:1} are (a)
$\min\{2,2,2,3\} =2$; (b) $\min\{3,4,5\} =3$.

\subsection{Comparisons of Total Depth with Stanley Depth}
\label{sec:comparison}
When  $S=S_k$, comparing \eqref{eq:sdepth S_k} with \eqref{eq:tdepth P}, we deduce that 
\begin{equation}\label{eq:S_k}
{\sdepth}\,\llbracket S_k\rrbracket =  {\tdepth}\,\llbracket S_k\rrbracket,
\end{equation}
i.e., the Stanley depth and the total depth of a normal set coincide.
Further, we will show that the Stanley depths of a general multiset $S$ and its base set coincide (Theorem \ref{thm:n=s}).

However, the Stanley depth and the total depth are generally {\it not} equal. 
In fact, for the multiset $S$ of \eqref{eq:S}, plugging \eqref{eq:height} into \eqref{eq:tdepth P} with $P=\llbracket S\rrbracket$ yields
\begin{equation}\label{eq:tdepth S}
{\tdepth}\,\llbracket S\rrbracket= \max _{\pi\in \mathfrak{P}(S)} \min_{[X,Y]\in \pi} \sum_{j=1}^k y_{j}\,,
\end{equation}
which is different from the right-hand side of definition \eqref{eq:sdepthS} of ${\sdepth}\,\llbracket S\rrbracket$.
In particular, from \eqref{eq:sdepthS} we obtain ${\sdepth}\,\llbracket S\rrbracket\le k$, while \eqref{eq:tdepth S} implies that ${\tdepth}\,\llbracket S\rrbracket$ could be greater than $k$. 
For instance, in Fig.\,\ref{fig:1}(b) we see that ${\tdepth}\,\llbracket 2,3\rrbracket \ge 3$. 
The total depth does not agree with the Stanley depth of a poset also for subposet of the subset lattice, see for instance \cite{GLS,KSSY}.
Sections \ref{sec:n^k}--\ref{sec:k=5} will provide many other concrete examples.

\section{Stanley Depth for Multisets}\label{sec:sdepthS}
In this section we show that the Stanley depths of the multiset $S$ of \eqref{eq:S} and its base set $\{1,2,\dots,k\}$  coincide,
by establishing a chain of two mappings (see \eqref{eq:pi_k} and \eqref{eq:pi*} below) from 
$\mathfrak{P}(S)$ to $\mathfrak{P}(S_k)$ and then bijectively to $\mathfrak{P}^*(S)$, defined below as the set of interval partitions of $\llbracket S\rrbracket$ with a special form \eqref{eq:f}. 
This approach will continue playing an important role in subsequent sections.
This result first appeared in Cimpoea\c s \cite[Theorem 1.3]{C}, but with an entirely different proof, by induction on the multiplicity $n_1$.

\begin{thm}\label{thm:n=s}
For $S=\{1^{n_1}, 2^{n_2}, \dots, k^{n_k}\}$ and $S_k=\{1,2,\dots,k\}$, we have
$$ 
{\sdepth}\,\llbracket S\rrbracket = {\sdepth}\,\llbracket S_k\rrbracket =  {\tdepth}\,\llbracket S_k\rrbracket.
$$ 
\end{thm}
\begin{proof}
We have proved the second equality; see \eqref{eq:S_k}.
To show the first equality,
for each interval $J=[X,Y]$ in $\llbracket S\rrbracket$,
we call the intersection 
\begin{equation}\label{eq:g}
J\cap \llbracket S_k\rrbracket=[X,g(Y)],\  {\rm{with}}\ \, g(Y):=(y_1 \wedge 1,\dots,y_k \wedge 1)\ {\rm{for}}\ Y=(y_1,\dots,y_k),
\end{equation}
the {\it induced interval} of $J$ in $\llbracket S_k\rrbracket$, where $\wedge$ represents the minimum function; 
in particular, if $x_j>y_j \wedge 1$ for some $j$, then $J\cap \llbracket S_k\rrbracket=\emptyset$. 

For each interval partition $\pi\in \mathfrak{P}(S)$, we consider its {\it induced interval partition} of $\llbracket S_k\rrbracket$ with empty intervals removed:
\begin{equation}\label{eq:pi_k:=}
\pi_k := \big\{J\cap \llbracket S_k\rrbracket: J\in \pi\  \text{and}\ J\cap \llbracket S_k\rrbracket \neq \emptyset \big\} 
\end{equation}
(with the partition property verified right below); in other words, we construct the following mapping:
\begin{equation}\label{eq:pi_k}
\pi=\bigcup_i\, [X_i,Y_i]\  
\longmapsto\ \pi_k=\bigcup_{X_i\le g(Y_i)}\, [X_i,g(Y_i)] \in \mathfrak{P}(S_k).
\end{equation}
For instance, the induced interval partition in $ \llbracket S_2\rrbracket$ of the partition in Fig.\,\ref{fig:1}(b) is $[01,01]\cup[10,11]$.

Here to show that the $\pi_k$ defined by \eqref{eq:pi_k:=} is indeed an interval partition of $\llbracket S_k\rrbracket$, 
we observe that  for any $J_1,J_2\in \pi\in \mathfrak{P}(S)$, the intersection of their induced intervals in $\llbracket S_k\rrbracket$: 
\begin{equation*}
\big( J_1 \cap \llbracket S_k\rrbracket \big) \cap \big( J_2 \cap \llbracket S_k\rrbracket \big) 
\subseteq J_1 \cap  J_2 = \emptyset,
\end{equation*}
and therefore, the intervals in  $\pi_k$ are pairwise disjoint.
Further, we  have the union of elements in $\pi_k$:
\begin{equation*}
\bigcup_{J\in \pi,\, J\cap \llbracket S_k\rrbracket \neq \emptyset } \big(J\cap \llbracket S_k\rrbracket \big) 
= \bigcup_{J\in \pi} \big(J\cap \llbracket S_k\rrbracket \big) 
= \left(\bigcup_{J\in \pi} J\right) \cap \llbracket S_k\rrbracket
= \llbracket S\rrbracket \cap \llbracket S_k\rrbracket
= \llbracket S_k\rrbracket,
\end{equation*}
as desired.

This $\pi_k$  has Stanley depth no smaller than that of $\pi$.
In fact, applying definition \eqref{eq:sdepth pi} to $\pi_k$ and $\pi$ gives
\begin{align*} 
{\sdepth}\,\pi_k
=&\ \min_{[X,Y']\in \pi_k}\! \#\{j: y'_{j}=1\}
\ge \min_{Y'=g(Y),\,[X,Y]\in \pi}\!\#\{j: y'_{j}=1\}\nonumber \\ 
=&\ \min_{[X,Y]\in \pi}\! \#\{j: y_j \wedge 1=1\}
\ge \min_{[X,Y]\in \pi} \#\{j: y_{j}=n_j\}\ =\ {\sdepth}\,\pi,
\end{align*}
as $n_j \wedge 1=1$.
Therefore
\begin{equation*}
{\sdepth}\,\llbracket S\rrbracket= \max_{\pi\in \mathfrak{P}(S)} {\sdepth}\,\pi
\le \max_{\pi\in \mathfrak{P}(S)}  {\sdepth}\,\pi_k\le {\sdepth}\,\llbracket S_k\rrbracket.
\end{equation*}

On the other hand, we say that an interval partition $\pi^*$ of $\llbracket S\rrbracket$ is {\it good} if each interval in $\pi^*$ has the form
\begin{equation}\label{eq:f}
[X,f(Y)]\ \ {\rm{with}}\ \ f(Y):=(n_1 y_1,\dots,n_k y_k)\ \ {\rm and}\ \ X,Y\in \{0,1\}^k,
\end{equation}
where $\{0,1\}^k\! :=\{(v_1,\dots, v_k): v_j \in \{0,1\}\}$;
namely, the components of the bottom element $X$ are either $0$ or $1$ and for each $j$, the $j$-th component of the top element $f(Y)$ is either $0$ or $n_j$ 
(in other words, every interval in $\pi^*$ is tall, in the sense that its bottom element  is a multiset with all multiplicities $0$ or $1$ and its top element is a multiset in which each multiplicity is $0$ or as large as possible). 
Then we have a bijection between $\mathfrak{P}(S_k)$ and the set of good interval partitions of $\llbracket S\rrbracket$, denoted by $\mathfrak{P}^*(S)$:
\begin{equation}\label{eq:pi*}
\pi_k=\bigcup_i\, [X_i,Y_i]\
\longmapsto\ \pi^*=\bigcup_i\, [X_i,f(Y_i)]
\end{equation}
(with the partition property verified right below). 
For instance, the induced interval partition $[01,01]\cup[10,11]$ in $\llbracket S_2\rrbracket$ of the partition in Fig.\,\ref{fig:1}(b) maps to $[01,03]\cup[10,23]$. 

Here to show that the $\pi^*$ defined in \eqref{eq:pi*} is indeed an interval partition of $\llbracket S\rrbracket$, 
we recall the order-preserving function $g(\cdot)$ from \eqref{eq:g}.
For any $[X_i,Y_i]\in \pi_k \in \mathfrak{P}(S_k)$ and 
$Z\in [X_i,f(Y_i)]$, we have
\begin{equation*}
g(Z)\in [g(X_i),g(f(Y_i))] = [X_i,Y_i].
\end{equation*}
Since the intervals $[X_i,Y_i]$ in $\pi_k$  are pairwise disjoint,  the intervals $[X_i,f(Y_i)]$ in  $\pi^*$ are pairwise disjoint as well. 
Further, for any $Z\in \llbracket S\rrbracket$, 
we notice that $g(Z)\in \llbracket S_k\rrbracket$, and thus
$g(Z)\in [X_i,Y_i]$ for some interval $[X_i,Y_i] \in \pi_k$. 
Recall from definition \eqref{eq:f} that the function $f(\cdot)$ is order-preserving. It follows that
\begin{equation*}
X_i\le g(Z) \le Z\le f(g(Z))\le f(Y_i);
\end{equation*}
in other words, we obtain $Z\in [X_i,f(Y_i)]\in \pi^*$.
Hence the union of elements in $\pi^*$ is $\llbracket S\rrbracket$, as desired.

Since the mapping \eqref{eq:pi*} keeps the Stanley depth unchanged, we deduce that
\begin{equation*}
{\sdepth}\,\llbracket S\rrbracket 
\ge \max_{\pi^*\in \mathfrak{P}^*(S)} {\sdepth}\,\pi^* 
=\max_{\pi_k\in \mathfrak{P}(S_k)} {\sdepth}\,\pi_k
= {\sdepth}\,\llbracket S_k\rrbracket
\end{equation*}
and complete the proof. 
\end{proof}

\section{Optimal Interval Partitions}
\label{sec:optimal}
Recall  that an interval partition of $\llbracket S\rrbracket$ is {\emph {good}} if each of its intervals has the form of \eqref{eq:f}.
In this section we prove that there exists a good interval partition with the maximal total depth (Theorem \ref{thm:good}) by adopting  the mappings \eqref{eq:pi_k} and \eqref{eq:pi*}. 
Therefore, in order to find ${\tdepth}\,\llbracket S\rrbracket$, it suffices to consider only the good interval partitions of $\llbracket S\rrbracket$. 

This result is crucial in the rest of this paper to finding (the bounds of) total depths.
In particular, if we know all the interval partitions of $\llbracket S_k\rrbracket$, we can transform them into good interval partitions of $\llbracket S\rrbracket$ via \eqref{eq:pi*}, and find ${\tdepth}\,\llbracket S\rrbracket$ by selecting the optimal from these good interval partitions.
This result also leads to the nondecreasing property of $\tdepth\,\llbracket n_1,\dots,n_k \rrbracket$ in every $n_i$ (Theorem \ref{thm:increase}).

\begin{defn}\label{def:optimal}
We say that an interval partition $\pi$ of $\llbracket S\rrbracket$ is  \emph{optimal} if ${\tdepth}\,\pi={\tdepth}\,\llbracket S\rrbracket$.
An optimal interval partition exists for any $\llbracket S\rrbracket$ since $\llbracket S\rrbracket$ is finite. 
\end{defn}

\begin{thm}\label{thm:good}
There exists an optimal interval partition of $\llbracket S\rrbracket$ such that it is good, i.e., each of its intervals has the form of \eqref{eq:f}.
\end{thm}
\begin{proof}
Let $\pi$ be an optimal interval partition of $\llbracket S\rrbracket$.
Recall from \eqref{eq:pi_k} and \eqref{eq:pi*} the mappings (with $f$ and $g$ defined in \eqref{eq:f} and \eqref{eq:g}, respectively)
\begin{equation*}
\pi=\bigcup_i\, [X_i,Y_i]\
\longmapsto\ \pi_k=\bigcup_{X_i\le g(Y_i)} [X_i,g(Y_i)]\
\longmapsto\ \pi^*=\bigcup_{X_i\le g(Y_i)} [X_i,f(g(Y_i))]
\end{equation*}
from $\pi$ to its induced interval partition $\pi_k$ of $\llbracket S_k\rrbracket$, and then to $\pi_k$'s corresponding good interval partition $\pi^*$ of $\llbracket S\rrbracket$. 
We show that $\pi^*$ is optimal, i.e., ${\tdepth}\,\pi^* = {\tdepth}\,\llbracket S\rrbracket$, or equivalently,
${\tdepth}\,\pi \le {\tdepth} \ \pi^*$.

Consider the top elements of the intervals in $\pi$, $\pi_k$ and $\pi^*$. 
For each $[X,Y]\in \pi$ with $X\le g(Y)=:Y'=(y'_1,y'_2,\dots,y'_k)$,  where $y'_j=1\wedge y_j $ for all $j$,
if $y_j\ge1=y_j'$, then $y_j \le n_j=n_j y'_j$\,; otherwise, we have $y_j=0=y'_j=n_j y'_j$. 
Thus ${y_j}\le n_jy'_j$ always holds.
Combining this with  \eqref{eq:height} yields
$$
{\height}\,Y=\sum_{j=1}^k {y_j}\le \sum_{j=1}^k n_jy'_j= {\height}\,f(Y')={\height}\,f(g(Y)).
$$
Therefore by the definition \eqref{eq:tdepth pi} of total depth, we obtain 
\begin{align*}
{\tdepth}\,\pi =&\ \min_{[X,Y]\in \pi} {\height}\,Y \le \min_{[X,Y]\in \pi,\,X\le g(Y)} {\height}\,Y\\
\le&\ \min_{[X,Y]\in \pi,\,X\le g(Y)} {\height}\,f(g(Y))= {\tdepth} \ \pi^*,
\end{align*}
as desired.
\end{proof}

Thanks to Theorem \ref{thm:good}, we shall assume that all the partitions in the remainder of this paper are good.
We now show the following monotonicity result.

\begin{thm}\label{thm:increase}
{\rm (i)} The total depth of $\llbracket n_1,n_2,\dots,n_k \rrbracket$ is nondecreasing in every $n_i$, $i\in \{1,2,\dots,k\}$.
It follows that if $n_i\le n_i'$ for all $i\in \{1,2,\dots,k\}$, 
then we have 
$
{\tdepth}\,\llbracket n_1,n_2,\dots,n_k \rrbracket \le {\tdepth}\,\llbracket n_1',n_2',\dots, n_k'\rrbracket.
$

\smallskip
\noindent
{\rm (ii)} For $k\ge 2$, we have 
$
{\tdepth}\,\llbracket n_2,\dots,n_k\rrbracket\le {\tdepth}\,\llbracket n_1,n_2,\dots,n_k \rrbracket.
$
\end{thm}
\begin{proof}
(i) We adopt the notation $\llbracket S\rrbracket =\llbracket n_1,n_2,\dots,n_k \rrbracket$. 
On the strength of Theorem \ref{thm:good}, bijection \eqref{eq:pi*} and definitions \eqref{eq:tdepth pi}, \eqref{eq:height}  and \eqref{eq:f}, we obtain
\begin{align*}
{\tdepth}\,\llbracket S\rrbracket 
&= \max_{\pi^*\in \mathfrak{P}^*(S)} {\tdepth}\,\pi^* 
= \max_{\pi_k\in \mathfrak{P}(S_k)} 
\min_{[X,Y]\in \pi_k} {\height}\,f(Y) \\
&=\ \max_{\pi_k\in \mathfrak{P}(S_k)} 
\min_{[X,Y]\in \pi_k} \sum_{j=1}^k n_j y_{j},
\end{align*}
which is nondecreasing in every $n_i$. 

(ii) Let  $\pi=\bigcup_i\, [X_i,Y_i]$ be an optimal interval partition of $\llbracket n_2,\dots,n_k\rrbracket$. 
Note that $\pi':=[\{1\}, S] \bigcup \left(\bigcup_i\, [X_i,Y_i]\right)$ is an interval partition of $\llbracket S\rrbracket$ and has the same total depth as $\pi$ since $\height\, S\ge \height\, Y_i$ for all $i$. 
The conclusion then follows from Definitions \ref{def:optimal} and \ref{def:tdepth} (equation \eqref{eq:tdepth P}).
\end{proof}





\section{\texorpdfstring{The Total Depth of $\bm{n}^k\,\backslash \{\bm{0}\}$ and Bounds for Total Depth}{The Total Depth of n\^{}k\textbackslash \{0\} and Bounds for Total Depth}}
\label{sec:n^k}

Recall that 
${\sdepth}\,\llbracket S_k\rrbracket=\lceil {k/ 2}\rceil$ (\cite{B}), and therefore  ${\tdepth}\,\llbracket S_k\rrbracket=\lceil {k/ 2}\rceil$ by \eqref{eq:S_k}.
\begin{thm}{\cite[Theorem 2.2]{B}}\label{thm:2^n}
For $k\ge 1$, 
we have  
$${\tdepth}\,\bm{2}^k\,\backslash \{\bm{0}\} ={\tdepth}\,\llbracket 1,1,\dots,1\rrbracket=\lceil {k/ 2}\rceil.$$
\end{thm}
We take advantage of this result to derive a pair of upper and lower bounds for the total depth for a general multiset. 

\begin{thm}\label{thm:k/2}
For $n_1 \le n_2 \le \cdots \le n_k$, we have
\begin{equation*}
\sum_{i=1}^{\lceil {k/ 2}\rceil} n_i \le {\tdepth\,} \llbracket n_1,n_2,\dots,n_k \rrbracket \le \sum_{i=k-\lceil {k/ 2}\rceil+1}^k n_i\,.
\end{equation*}
\end{thm}

When all $n_i$'s equal $n-1$, 
the upper and lower bounds coincide, which gives the total depth of
$\bm{n}^k\,\backslash \{\bm{0}\}$.

\begin{cor}\label{coro:n^k}
For $n\ge 2$, we have  
$${\tdepth}\ \bm{n}^k\,\backslash \{\bm{0}\} 
= {\tdepth}\, \llbracket n-1,n-1,\dots,n-1\rrbracket
=(n-1)\lceil {k/ 2}\rceil.$$
\end{cor}

\begin{proof}[Proof of Theorem \ref{thm:k/2}]
{\it Upper bound.}
Let $\pi^*$ be an optimal interval partition of $\llbracket S\rrbracket$. 
Thanks to Theorem \ref{thm:good}, we can assume that $\pi^*$ is good with $\pi^*=\{[X,f(Y)]:[X,Y]\in \pi_k\}$ as in \eqref{eq:pi*}, 
where $\pi_k$ is its induced interval partition of $\llbracket S_k\rrbracket$. 
Since ${\tdepth\,} \pi_k\le {\tdepth\,} \llbracket S_k\rrbracket= \lceil {k/ 2}\rceil$ by \eqref{eq:tdepth P} and Theorem \ref{thm:2^n}, 
it follows from \eqref{eq:tdepth pi} 
that there exists $Y'\!$ with $[X,Y']\in \pi_k$ such that ${\height}\,Y'\! \le \lceil {k/ 2}\rceil$, i.e., at most $\lceil {k/ 2}\rceil$ of the $y'_i$'s are $1$ (recall \eqref{eq:depthS_k}). 
Thus 
$${\tdepth}\,\llbracket S\rrbracket 
= {\tdepth}\,\pi^*
\le {\height}\,f(Y')
=\sum_{i=1}^k n_i y'_i
\le \sum_{i=k-\lceil {k/ 2}\rceil+1}^k n_i\,.$$

{\it Lower bound.}
Let $\pi_k$ be an optimal interval partition of $\llbracket S_k\rrbracket$ and $\pi^*=\{[X,f(Y)]:[X,Y]\in \pi_k\}$ its corresponding good interval partition of $\llbracket S\rrbracket$ as in \eqref{eq:pi*}. 
Thanks to Theorem \ref{thm:2^n}, we have 
${\tdepth\,} \pi_k={\tdepth\,} \llbracket S_k\rrbracket=\lceil {k/ 2}\rceil$. 
Thus ${\height}\,Y\ge \lceil {k/ 2}\rceil$ for any interval $[X,Y]\in \pi_k$ by \eqref{eq:tdepth pi}, i.e., 
at least $\lceil {k/ 2}\rceil$ of the $y_i$'s are $1$  (recall \eqref{eq:depthS_k}). 
Therefore
${\height}\,f(Y)=\sum_{i=1}^k n_i y_i\ge \sum_{i=1}^{\lceil {k/ 2}\rceil} n_i$. 
Hence 
$$\qquad \qquad \quad 
{\tdepth}\,\llbracket S\rrbracket \ge {\tdepth}\,\pi^*
=\min_{[X,Y]\in \pi_k}  {\height}\,f(Y)
\ge\sum_{i=1}^{\lceil {k/ 2}\rceil} n_i\,.\qquad \qedhere
$$
\end{proof}

\section{More Bounds for Total Depth}\label{sec:bounds}

In this section we derive more bounds for the total depth. They are stronger than those in Theorem \ref{thm:k/2} in many (but not all) cases, for instance, when $k$ is odd or $n_k$ is big, or when $k$ is even and $n_1,\dots,n_{k-1}$ are not too close to each other.
We prove our results by analyzing the top elements of the intervals that contain one of these submultisets: $\{1\},\dots,\{k\}$, i.e.,
$(1,0,\dots,0),\dots,(0,\dots,0,1)$.
This approach also plays an important role in computing total depths in Section \ref{sec:k=5}.

\begin{thm}\label{thm:n_k}
For $1\le n_1\le n_2\le \cdots \le n_k$, we have
\begin{equation}\label{eq:n_k}
n_k \le {\tdepth\,} \llbracket n_1,n_2,\dots,n_k \rrbracket 
\le \max\{n_k, n_1+n_2+\cdots+n_{k-1}\},
\end{equation}
and
\begin{equation}\label{eq:in_i}
{\tdepth\,} \llbracket n_1,n_2,\dots,n_k \rrbracket \le \left\lfloor{\frac{1}{k}\sum_{i=1}^k i n_i}\right\rfloor.
\end{equation}
\end{thm}

\begin{rmk}[Comparison of the bounds]
None of the upper or lower bounds in \eqref{eq:n_k}, \eqref{eq:in_i} and Theorem \ref{thm:k/2} is always stronger than another. 

In fact, for the upper bounds in \eqref{eq:in_i} and Theorem \ref{thm:k/2},  
one can verify that when $k$ is odd, the upper bound in \eqref{eq:in_i} is stronger (i.e., no greater).
However, when $k$ is even, neither of them is always stronger; the upper bound in \eqref{eq:in_i} is stronger only if $n_1,n_2,\dots,n_{k-1}$ are not too close to each other.
In particular, when all $n_i$'s equal $n$, 
the inequality \eqref{eq:in_i} becomes ${\tdepth}\,\llbracket S\rrbracket \le \lfloor {{n(k+1)}/ 2}\rfloor$,
which is an equality if and only if $k$ is odd or $n=1$ (Corollary \ref{coro:n^k}).
This means that when $k$ is even and $n_i= n\ge 2$ for all $i$, 
the upper bound in \eqref{eq:in_i} is weaker.

Similarly, comparing the bounds in \eqref{eq:n_k} with others in the case of all equal $n_i$'s (or of similar size) and the case of $n_k\ge n_1+n_2+\cdots+n_{k-1}$, we find that no bound is always stronger than another. \qed
\end{rmk}

Before proving Theorem \ref{thm:n_k}, we need the following lemma, which allows us to derive the upper bounds in Theorem \ref{thm:n_k} and later in Section \ref{sec:k=5}.

\begin{lem}\label{lemma1}
In a given good interval partition of $\llbracket n_1,n_2,\dots,n_k\rrbracket$,
let $I(i)$ with top element $t(i)$ be the 
interval with bottom element $\{i\}$,  $i=1,2,\dots ,k$. 
Then the $I(i)$'s are distinct and the components $t(i)_j$ $(j=1,2,\dots,k)$ of $t(i)$ satisfy $$t(j)_i=0 \quad \text{for\ any}\ \ i\neq j\ \ \text{with}\ \ t(i)_j>0.$$ 
\end{lem}
\begin{proof}
Since each of $\{1\},\{2\},\dots,\{k\}$ is no greater than any other element of  $\llbracket n_1,n_2,\dots,n_k\rrbracket$, they must be the bottom elements of  distinct intervals. 
If $t(i)_j,t(j)_i>0$ for some $i\neq j$,
then $\{i,j\}\in I(i)\cap I(j)$.
This contradicts the fact that $I(i)$ and $I(j)$ are distinct and thus disjoint. 
\end{proof}

\begin{rmk}
By the definition \eqref{eq:f} of good interval partition, we have
\begin{equation}\label{eq:t(i)_j}
t(i)_j\in\{0,n_j\}\quad \text{and}\quad t(i)_i = n_i.
\end{equation}
\end{rmk}

Since we only consider good interval partitions,  for convenience we write an interval in a good interval partition of $\llbracket S\rrbracket$ by its induced interval in $\llbracket S_k\rrbracket$.
For example, $[10001, n_1n_200n_5]$ is written as $[10001,11001]$. 
In case the top and bottom elements of the induced interval coincide, we omit one of them. 
For instance, we write $[10001,n_1000n_5]$ as $[10001]$.

\begin{proof}[Proof of Theorem \ref{thm:n_k}]
(i) The lower bound  in \eqref{eq:n_k} follows from the observation that the following interval partition  has total depth $n_k$:
\begin{equation}\label{eq:partition}
[00\cdots01]\, \bigcup \left(\bigcup\,[x_1x_2\cdots x_{k-1}0, x_1x_2\cdots x_{k-1}1]\right),
\end{equation} 
where the second union is taken over $(x_1,x_2,\dots, x_{k-1})\in \{0,1\}^{k-1}\backslash\{(0,0,...,0)\}$.

\smallskip
(ii) 
We prove the upper bound  in \eqref{eq:n_k} by contradiction. Assume the contrary 
that there exists a good interval partition $\pi$ of  $\llbracket n_1,n_2,\dots,n_k\rrbracket$ 
such that ${\tdepth}\,\pi > \max\{n_k,  n_1+n_2+\cdots+n_{k-1}\}$. 
Thus from \eqref{eq:height} and \eqref{eq:tdepth pi} we get
\begin{equation}\label{eq:t(i)}
\sum_{j=1}^k  t(i)_j={\height\,} t(i)\ge{\tdepth}\,\pi > \max\{n_k, n_1+n_2+\cdots+n_{k-1}\}
\end{equation}
for all $1\le i\le k$.
It then follows from Lemma \ref{lemma1} that $t(i)_k>0$ for all $i$ and $t(k)_i=0$ for all $1\le i\le k-1$. 
Therefore ${\height}\,t(k)= t(k)_k \le n_k$, which contradicts \eqref{eq:t(i)} with $i=k$.

\smallskip
(iii) For \eqref{eq:in_i}, apply Lemma \ref{lemma1} to an optimal (and good) interval partition:
\begin{equation}\label{eq:aij+aji}
t(i)_j + t(j)_i \le \max\{n_i, n_j\}=n_{\max\{i,j\}},\quad \forall\ i\neq j.
\end{equation}

Summing \eqref{eq:aij+aji} and the second equality of \eqref{eq:t(i)_j} over $i,j=1,2,\dots, k$, along with \eqref{eq:height}, we yield
\begin{align*}
\sum_{i=1}^k {\height\,} t(i) 
=&\ \sum_{i,j=1}^k  t(i)_j 
= \sum_{1\le i<j\le k}  \big(t(i)_j + t(j)_i\big) + \sum_{i=1}^k  t(i)_i\\
\le&\ \sum_{1\le i<j\le k} n_{\max\{i,j\}} +\sum_{i=1}^k  n_i
= \sum_{i=1}^k i n_i.
\end{align*}
Thus there exists an $i$ such that ${\height\,} t(i)$ is at most $\lfloor{\sum_{i=1}^k i n_i/k}\rfloor$.
Therefore, the total depth of this interval partition is at most $\lfloor{\sum_{i=1}^k i n_i/k}\rfloor$ by \eqref{eq:tdepth pi}, which implies \eqref{eq:in_i}.
\end{proof}

\begin{rmk}\label{rmk:n_k}
When $n_k\ge n_1+n_2+\cdots+n_{k-1}$, the inequality \eqref{eq:n_k} becomes an equality: 
${\tdepth\,} \llbracket n_1,$ $n_2,\dots,n_k \rrbracket = n_k$. 
This implies that the interval partition  \eqref{eq:partition} is in fact optimal.
\qed
\end{rmk}

\section{\texorpdfstring{Case of $k\le 5$}{Case of k<=5}}\label{sec:k=5}
In this section we determine the total depth of $\llbracket n_1,n_2,\dots,n_k \rrbracket$ for $k\le 5$ using Theorems \ref{thm:good}, \ref{thm:n_k} and Lemma \ref{lemma1}.
For ease of notation, we denote ${\height\,} t(i)$ by $h(i)$ (recall $t(i)$ from Lemma \ref{lemma1}), and $\sum_{i\in H} n_i$ by $\sigma(H)$. 
For example, we have $\sigma(13)=n_1+n_3$ and $\sigma(\emptyset)=0$.

\begin{thm}\label{thm:5}
For $1\le n_1\le n_2\le \cdots \le n_5$, we have
\begin{align*}
\tdepth\,\llbracket n_1 \rrbracket &=\sigma(1),\\
\tdepth\,\llbracket n_1,n_2 \rrbracket &=\sigma(2),\\
\tdepth\,\llbracket n_1,n_2,n_3 \rrbracket &=\max\{\sigma(3),\sigma(12)\},\\
\tdepth\,\llbracket n_1,n_2,n_3,n_4 \rrbracket &=\max\{\sigma(4),\sigma(24)\wedge \sigma(123)\},\\
\tdepth\,\llbracket n_1,n_2,n_3,n_4,n_5 \rrbracket 
&= \max\{
\sigma(5),
\sigma(35) \wedge \sigma(1234),
\sigma(45)\wedge \sigma(234)\wedge \sigma(135), \\
&\ \ \ \ \ \ \  \ \ \ \,
\sigma(45)\wedge \sigma(1234)\wedge \sigma(125),
\sigma(125)\wedge \sigma(134)
\},
\end{align*}
where $\wedge$ represents the minimum function. 
\end{thm}

\begin{proof}
Recall from Theorem \ref{thm:good} that it suffices to consider only the good interval partitions.
For convenience, we denote the right-hand side by $d_k$ for the case $k$.

\medskip
{\bf (1) {$k=1$}.}
The total depth of the only good interval partition $[1]$ is $n_1$.

\medskip
{\bf(2) {$k=2$}.}
By Remark \ref{rmk:n_k}, the total depth is $n_2$, achieved by $[10,11]\cup [01]$.

\medskip
{\bf(3) {$k=3$}.}
The total depth is at most $d_3$ by \eqref{eq:n_k}, and $d_3$ is achieved by one of the following partitions.

\smallskip
\noindent
{\it Fact  3.1.}
${\tdepth}\, [110,111]\cup[100,101]\cup[010,011]\cup[001]=\sigma(3).$

\smallskip
\noindent
{\it Fact  3.2.}
${\tdepth}\, [100,110]\cup[001,101]\cup[010,011]\cup[111]= \sigma(12).$

\medskip
{\bf(4) {$k=4$}.}
Computing the total depths of the following two partitions shows that $d_4$ can be achieved.

\smallskip
\noindent
{\it Fact  4.1.}
${\tdepth}\, \pi=\sigma(4)$ for $\pi=
[1100,1111]\cup[1000,1011]\cup[0100,0111]\cup[0010,0011]\cup[0001].$

\smallskip
\noindent
{\it Fact  4.2.}
${\tdepth} \, \pi=\sigma(24)\wedge \sigma(123)$ for 
$\pi=
[1000,1011]\cup[0100,1110]\cup[0010,0011]\cup[0001,0101]\cup[1101,1111]\cup[0111].$

\smallskip
Now we show that the total depth is at most $d_4$.

\smallskip
\noindent
{\bf Case 1.}
$\sigma(24) \ge \sigma(123)$, i.e., $d_4=\max\{\sigma(4),\sigma(123)\}$. The conclusion follows by
\eqref{eq:n_k}.

\smallskip
\noindent
{\bf Case 2.}
$\sigma(24)< \sigma(123)$.
We prove by contradiction.
Assume the contrary that $\pi$ is a good interval partition such that 
$d_4<{\tdepth}\,\pi \le h(i)$ 
for all $i$
(recall definition \eqref{eq:tdepth pi}).
Then  
$h(i)>d_4=\sigma(24)$ for all $i$. 


\smallskip
(i) If $t(4)_3=0$, then $t(4)_1,t(4)_2>0$ since $h(4)>\sigma(24)$. 
It follows from Lemma \ref{lemma1} that $t(1)_4=t(2)_4=0$.
Since $h(1)>\sigma(24)\ge\sigma(13)$, we have $t(1)_2>0$ and hence $t(2)_1=0$.
Therefore $h(2)\le \sigma(23)\le d_4$, which is a contradiction.

\smallskip
(ii) If $t(4)_3>0$, then  $t(3)_4=0$ by Lemma \ref{lemma1}.
Since $h(3)>\sigma(24)$, we have $t(3)_1,t(3)_2>0$. Thus $t(1)_3=t(2)_3=0$.
Combining this with $h(2)>\sigma(24)$ leads to $t(2)_1>0$. Thus $t(1)_2=0$. Hence $h(1)\le\sigma(14)\le\sigma(24)=d_4$, which is a contradiction.

\medskip
{\bf(5) {$k=5$}.} The idea is similar while the proof becomes much more complicated; see \ref{sec:5} for details.
\end{proof}

\section{Future Research}
From the case $k=5$, we see that the situation will become extremely complicated when $k\ge 6$ via the same approach.
It would be interesting to know if there is an explicit expression for the total depth for any $k$.

It would also be interesting to investigate other classes of posets to see if their total depths can be found through a combinatorial approach, for example, the poset $\llbracket V_{q,k}\rrbracket$ consisting of the nontrivial subspaces of a $k$-dimensional vector space $V_{q,k}$ would be a $q$-analogue of $\llbracket S_k\rrbracket$. 
Another class of posets of interest, generalizing the poset $\llbracket V_{q,k}\rrbracket$, would be the geometric lattices excluding the bottom element.


Finally, if we replace ${\height}\,Y$ in the right-hand side of (\ref{eq:tdepth pi}) with the \emph{interval depth} of  $[X,Y]$, namely the maximal length of a chain from $X$ to $Y$, we would obtain another variant 
of the Stanley depth. 
Through the same approach, we can prove that for the poset $\llbracket S\rrbracket$, there still exists a good interval partition with the maximal interval depth. 
It would be interesting to investigate this new variant 
for various classes of posets.

\section*{Acknowledgements} 
The author is greatly indebted to Professor Richard P. Stanley for suggesting this problem, for very careful readings of the previous versions of the paper, and for invaluable advice.
She is also grateful to the referees and the associate editor for their very careful and helpful comments.

\appendix
\section{\texorpdfstring{Proof of Theorem \ref{thm:5} for $k=5$}{Proof of Theorem \ref{thm:5} for k=5}}\label{sec:5}
\begin{proof} 
Computing the total depths of the following five interval partitions shows that the value on the right-hand side, denoted by $d_5$, can be achieved.

\smallskip
\noindent
{\it Fact  5.1.}
${\tdepth}\,\pi =\sigma(5)$ for
\begin{align*}
\pi=&\ [11000,11111]\cup[10000,10111]\cup[01000,01111]\cup[00100,00111]\\
&\ \cup[00010,00011]\cup[00001].
\end{align*}

\noindent
{\it Fact  5.2.}
${\tdepth}\,\pi  =\sigma(35)\wedge\sigma(1234)=: m_1$ for
\begin{align*}
\pi=&\ [10000,10101]\cup[01000,01101]\cup[00100,00101]\cup[00010,11110]\\
&\ \cup[00001,00011] \cup[01011,01111]\cup[10011,10111]\cup[11000,11101]\\
&\ \cup[11111]\cup[11011]\cup[00111].
\end{align*}

\noindent
{\it Fact  5.3.}
${\tdepth}\,\pi  =\sigma(45)\wedge \sigma(234)\wedge \sigma(135)=:  m_2$ for
\begin{align*}
\pi=&\ [00001,00011]\cup[00010,01110]\cup[00100,10101]\cup[01000,01101]\cup[10000,\\
&\ \, 10011] \cup[00111,01111]\cup[10110,10111]\cup[11000,11111]\cup[01011].
\end{align*}

\noindent
{\it Fact  5.4.}
${\tdepth}\,\pi  =\sigma(45)\wedge \sigma(1234)\wedge \sigma(125)=:  m_3$ for
\begin{align*}
\pi=&\ [00001,00011]\cup[00010,11110]\cup[00100,10101]\cup[01000,01101]\\
&\ \cup[10000,11001] \cup[10011,10111]\cup[01110,01111]\\
&\ \cup[11100,11101]\cup[11011,11111]\cup[00111].
\end{align*}

\smallskip
\noindent
{\it Fact  5.5.}
${\tdepth} \,\pi =\sigma(125)\wedge \sigma(134)=:  m_4$ for
\begin{align*}
\pi =&\ [10000,11001]\cup[01000,01110]\cup[00100,10101]\cup[00010,10110]\\
&\ \cup[00001,01011] \cup[10011,11011]\cup[01101,01111]\cup[11010,11110]\\
&\ \cup[00111,10111]\cup[11100,11101]\cup[11111].
\end{align*}

Now we show that  the total depth is at most $d_5$.
\begin{case}
$m_1=\!\sigma(1234)$, i.e., $d_5=\max\{\sigma(5),\sigma(1234)\}$. 
The conclusion follows from
\eqref{eq:n_k}.
\end{case}

\begin{case}
$m_1=\sigma(35)$. 
Then $m_2$ can be
$\sigma(234)$, $\sigma(135)$ or $\sigma(45)$.
We prove by contradiction.
Assume the contrary that $\pi$ is a good interval partition with
${\tdepth}\,\pi>d_5$. 
We will consider the $t(i)$'s and apply Lemma \ref{lemma1} repeatedly.

\begin{subcase}
$m_2=\sigma(234)$.
Then $m_3$ can be $\sigma(1234)$, $\sigma(45)$ or $\sigma(125)$.
\begin{subsubcase}
$m_3=\sigma(1234)$.
Then ${\tdepth}\,\pi>d_5\ge\max\{\sigma(5),\sigma(1234)\}$, 
which contradicts \eqref{eq:n_k}.
\end{subsubcase}

\begin{subsubcase}\label{2.1.2}
$m_3=\sigma(45)$. Then 
$h(i)>d_5=\max\{\sigma(45),\sigma(234)\}$ for all $i$.
Since $h(5)> \sigma(45)>\sigma(5)$,
there exists $j\le 4$ such that $t(5)_j>0$ and thus $t(j)_5=0$. 
Since $h(j)> \sigma(234)$, we get $t(j)_i>0$ for all $i\le 4$. 
Thus for $i\in \{1,2,3,4\}\backslash \{j\}$, we have $t(i)_j=0$. 
It follows that $t(i)_5>0$ (otherwise, $h(i)\le \sigma(1234) - t(i)_j \le \sigma(234)$) 
and $t(5)_i=0$. Hence $h(5)\le t(5)_j+t(5)_5\le \sigma(45)$, which is a contradiction.
\end{subsubcase}

\begin{subsubcase}
$m_3=\sigma(125)$.
Then 
$h(i)>d_5=\max\{\sigma(35),\sigma(234),\sigma(125)\}$ for all $i$.
If $t(j)_5=0$  for some $j\le 3$, then
it follows from the analysis in Subcase \ref{2.1.2} that $h(5)\le t(5)_j+t(5)_5\le \sigma(35)$, which is a contradiction.
Hence
$t(j)_5>0$ and thus $t(5)_j=0$ for all $j\le 3$.
Since $h(5)>\sigma(35)>\sigma(5)$, we have $t(5)_4>0$ and thus $t(4)_5=0$.
Since $h(4)>\sigma(234)$, we get $t(4)_i>0$ and thus $t(i)_4=0$ for all $i\le 3$.
Since $h(1),h(2)>\sigma(125)$, we have $t(1)_3,t(2)_3>0$. 
It follows that  $t(3)_1=t(3)_2=0=t(3)_4$.
Therefore $h(3)\le \sigma(35)$, which is a contradiction.
\end{subsubcase}
\end{subcase}

\begin{subcase}
$m_2=\sigma(135)$.
Then 
$h(i)>d_5=\sigma(135)$ for all $i$.
Since $h(3)>\sigma(135)\ge \sigma(123)$, at least one of $t(3)_4$ and $t(3)_5$ is positive. 

\begin{subsubcase}
$t(3)_4, t(3)_5>0$.
Then $t(4)_3=t(5)_3=0$ by Lemma \ref{lemma1}.
Thus 
\begin{equation}\label{eqt4354}
t(4)_3=0,\
h(4)>\sigma(135)\ge \sigma(124)\ 
\Rightarrow\ t(4)_5>0\ 
\Rightarrow\ t(5)_4=0.
\end{equation}
Therefore $h(5)\le\sigma(125)\le\sigma(135)$, which is a contradiction.
\end{subsubcase}

\begin{subsubcase}
$t(3)_4=0<t(3)_5$. 
Then $t(3)_2>0$ (otherwise, $h(3)\le \sigma(135)$) and $t(5)_3=t(2)_3=0$ (Lemma \ref{lemma1}).
Since $h(2)>\sigma(135)$, we have
$t(2)_4,t(2)_5>0$ and thus $t(4)_2=t(5)_2=0$.
Combining this with $h(4)>\sigma(135)\ge\sigma(134)$, $t(5)_3=0$,
and $h(5)>\sigma(135)\ge \sigma(15)$ yields $t(4)_5,t(5)_4>0$, 
which contradicts Lemma \ref{lemma1}.
\end{subsubcase}

\begin{subsubcase}
$t(3)_5=0<t(3)_4$. 
Then $t(4)_3=0$ (Lemma \ref{lemma1}) and $t(3)_2>0$ (otherwise, $h(3)\le \sigma(134)\le \sigma(135)$).
Therefore $t(2)_3=0$ (Lemma \ref{lemma1}) and \eqref{eqt4354} still holds, namely, $t(5)_4=0$. Since $h(5)>\sigma(135)$, we have 
$t(5)_2>0$ and thus $t(2)_5=0$.
Therefore $h(2)\le \sigma(124)\le \sigma(135)$, which is a contradiction.
\end{subsubcase}
\end{subcase}

\begin{subcase}
$m_2=\sigma(45)$.
Then $m_4$ can be $\sigma(125)$ or $\sigma(134)$.
\begin{subsubcase}
$m_4=\sigma(125)$.
Then 
$h(i)>d_5=\max\{\sigma(45),\sigma(125)\}$ for all $i$.
Since $h(4)>\sigma(125) \ge \sigma(124)$,  at least one of $t(4)_3$ and $t(4)_5$ is positive.

\begin{subsubsubcase}
$t(4)_5>0$. Then $t(5)_4=0$ by Lemma \ref{lemma1}. Thus we have
\begin{align}\label{eqt35}
t(5)_4=0,\ h(5)>\sigma(125)\ &\Rightarrow\ t(5)_3>0\ \Rightarrow\ t(3)_5=0,\\ \label{eqt43}
t(3)_5=0,\ h(3)>\sigma(125) \ge \sigma(123)\ 
&\Rightarrow\ t(3)_4>0\ \Rightarrow\ t(4)_3=0.
\end{align}
Since $t(5)_4=0$ and $h(5)>\sigma(45) \ge \sigma(35)$, at least one of $t(5)_1$ and $t(5)_2$ is positive.

\smallskip
(i) If $t(5)_1>0$, then $t(1)_5=0$ by Lemma \ref{lemma1}. Thus we have
\begin{align}\label{eqt41}
t(1)_5=0,\ h(1)>\sigma(125) \ge \sigma(123) 
\Rightarrow t(1)_4\!>0\, \Rightarrow&\ t(4)_1\!=0,\\ \nonumber 
t(4)_1=t(4)_3=0\ \text{(\ref{eqt43},\ \ref{eqt41})},\ h(4)>\sigma(45) \Rightarrow  t(4)_2\!>0\, \Rightarrow&\ t(2)_4\!=0,\\ \label{eqt32}
t(2)_4=0,\ h(2)>\sigma(125) \Rightarrow t(2)_3\!>0\, \Rightarrow&\ t(3)_2\!=0,\\ \label{eqt32513}
\!\!\! t(3)_2\!=\!t(3)_5\!=\! 0\, \text{(\ref{eqt35}, \ref{eqt32})},\, h(3)\! >\!\sigma(45)\! \ge \!\sigma(34) \Rightarrow t(3)_1\! >0\, \Rightarrow&\ t(1)_3\!=0. 
\end{align}
Combining this with $t(1)_5=0$ yields  $h(1)\le\sigma(124)\le\sigma(125)$, which is a contradiction.

\smallskip
(ii) If $t(5)_1=0<t(5)_2$, then $t(2)_5=0$ by Lemma \ref{lemma1}. Thus we have
\begin{align}\label{eqt3242}
\!\! t(2)_5=0,\, h(2)\!>\!\sigma(125) \!\ge\! \sigma(124) \Rightarrow t(2)_3,t(2)_4 >0\, \Rightarrow t(3)_2 = t(4)_2 = 0,
\end{align}
therefore \eqref{eqt32513} still holds.
Combining \eqref{eqt43}, \eqref{eqt3242} and $h(4)>\sigma(45) $ leads to $t(4)_1>0$ and thus $t(1)_4=0$. Therefore $h(1)\le \sigma(125)$ by \eqref{eqt32513}, which is a contradiction.
\end{subsubsubcase}

\begin{subsubsubcase}
$t(4)_3>0=t(4)_5$. Then $t(3)_4=0$ by Lemma \ref{lemma1}. 
Since $h(3)>\sigma(125)\ge \sigma(123)$, we have $t(3)_5>0$ and thus  $t(5)_3=0$.
Further, since $h(5)>\sigma(45)$, at least one of $t(5)_1$ and $t(5)_2$ is positive.

\smallskip
(i) If $t(5)_1>0$, then $t(1)_5=0$ by Lemma \ref{lemma1}. Thus \eqref{eqt41} still holds.
Combining this with $t(4)_5=0$ and $h(4)>\sigma(45) \ge \sigma(34) $ leads to $t(4)_2>0$ and thus $t(2)_4=0$. 
Therefore \eqref{eqt32} still holds and
\begin{equation}\label{eqt13}
t(3)_2=0=t(3)_4,\ h(3)>\sigma(45)\ge \sigma(35) \Rightarrow t(3)_1>0\ \Rightarrow t(1)_3=0.
\end{equation}
Combining this with $t(1)_5=0$ yields  $h(1)\le \sigma(124)\le\sigma(125)$, which is a contradiction.

\smallskip
(ii) If $t(5)_1=0<t(5)_2$, then $t(2)_5=0$  by Lemma \ref{lemma1}. 
Thus \eqref{eqt3242} and \eqref{eqt13} still hold.
Combining \eqref{eqt3242} with $t(4)_5=0$ and $h(4)>\sigma(45) \ge \sigma(34)$ leads to $t(4)_1>0$ and thus $t(1)_4=0$. Therefore $h(1)\le \sigma(125)$ by \eqref{eqt13}, which is a contradiction.
\end{subsubsubcase}
\end{subsubcase}

\begin{subsubcase}
$m_4=\sigma(134)$.
Then 
$h(i)>d_5=\max\{\sigma(45),\sigma(134)\}$ for all $i$.
Since $h(4)>\sigma(134)\ge \sigma(124)$, at least one of $t(4)_3$ and $t(4)_5$ is positive.

\begin{subsubsubcase}
$t(4)_5=0<t(4)_3$. Then $t(3)_4=0$ by Lemma \ref{lemma1}. Thus 
\begin{align}\label{eqt3453}
t(3)_4=0,\ h(3)>\sigma(134)\ge \sigma(123) 
\Rightarrow t(3)_5>0\, \Rightarrow&\ t(5)_3=0,\\ \label{eqt4524}
t(4)_5=0,\ h(4)>\sigma(134) \Rightarrow t(4)_2>0\, \Rightarrow&\ t(2)_4=0,\\ \label{eqt2452}
t(2)_4=0,\ h(2)>\sigma(134)\ge \sigma(123) 
\Rightarrow t(2)_5>0\, \Rightarrow&\ t(5)_2=0,\\ \nonumber
t(5)_2=t(5)_3=0\ \text{(\ref{eqt3453},\ \ref{eqt2452})},\ h(5)>\sigma(45)
\Rightarrow t(5)_1>0\,
\Rightarrow&\ t(1)_5=0,\\ \label{eqt2341}
t(1)_5=0,\ h(1)>\sigma(134) \Rightarrow t(1)_i\,>0,\ i=2,3,4\, \Rightarrow&\ t(i)_1\,=0,\\  \nonumber
t(3)_1=0=t(3)_4,\
h(3)>\sigma(45)\ge \sigma(35) 
\Rightarrow t(3)_2>0\, \Rightarrow&\ t(2)_3=0.
\end{align}
Combining this with \eqref{eqt4524} and \eqref{eqt2341} yields $h(2)\le\sigma(25) \le \sigma(45)$, which is a contradiction.
\end{subsubsubcase}

\begin{subsubsubcase}
$t(4)_5>0$. Then  $t(5)_4=0$ by Lemma \ref{lemma1}. Since $h(3)>\sigma(134)\ge \sigma(123)$, at least one of $t(3)_4$ and $t(3)_5$ is positive.

\smallskip
(i) If $t(3)_5>0$, then $t(5)_3=0$ by Lemma \ref{lemma1}. 
Since $t(5)>\sigma(45)\ge \sigma(25)$, we have $t(5)_1,t(5)_2>0$ and hence $t(1)_5=t(2)_5=0$.
Thus \eqref{eqt2341} still holds and 
\begin{align}\label{eqt342}
t(2)_5=0, h(2)>\sigma(134) 
\Rightarrow t(2)_3,t(2)_4>0  \Rightarrow t(3)_2= t(4)_2=&\ 0, \\ \nonumber
\!\! t(3)_2=0=t(3)_1\, \eqref{eqt2341},\, h(3)\!>\!\sigma(45)\!\ge\! \sigma(35) 
\Rightarrow t(3)_4>0\, \Rightarrow t(4)_3=&\ 0.
\end{align}
Combining this with \eqref{eqt2341} and \eqref{eqt342} yields $h(4)\le\sigma(45)$, which is a contradiction. 

\smallskip
(ii) If $t(3)_5=0<t(3)_4$, then $t(4)_3=0$ by Lemma \ref{lemma1}. Thus we have
\begin{align}\label{eqt3523}
t(3)_5=0,\ h(3)>\sigma(134)  \Rightarrow\, t(3)_2>0\ \Rightarrow&\ t(2)_3=0,\\
\nonumber
t(2)_3=0,\ h(2)>\sigma(134)\ge \sigma(124) 
\Rightarrow\, t(2)_5>0\ \Rightarrow&\ t(5)_2=0,\\
\nonumber
t(5)_2=0=t(5)_4,\ h(5)>\sigma(45)\ge \sigma(35) 
\Rightarrow\, t(5)_1>0\ \Rightarrow&\ t(1)_5=0,
\end{align}
therefore \eqref{eqt2341} still holds.
Combining this with \eqref{eqt3523} and $h(2)>\sigma(45)\ge \sigma(25)$ leads to $t(2)_4>0$ and thus  $t(4)_2=0$.  In conjunction with \eqref{eqt2341} and  $t(4)_3=0$,  we yield $h(4)\le\sigma(45)$, which is a contradiction. 
\end{subsubsubcase}
\end{subsubcase}
\end{subcase}
\end{case}
This completes the proof.
\end{proof}

\end{document}